\documentclass[11pt]{article}

\usepackage{amsmath,amsthm,amssymb,mathtools}
\usepackage{microtype}
\usepackage{hyperref}
\usepackage{authblk}
\usepackage{pgfplots}
\usepackage{float}
\usepackage{xurl}
\pgfplotsset{compat=1.18}
\usepackage[nameinlink,capitalize]{cleveref}

\newtheorem{theorem}{Theorem}[section]
\newtheorem{lemma}[theorem]{Lemma}

\theoremstyle{definition}

\newtheorem{remark}[theorem]{Remark}

\newcommand{\norm}[1]{\left\lVert #1 \right\rVert}

\DeclareMathOperator{\tr}{tr}

\title{Carbery's inequality in the Schatten--von Neumann classes}

\author[3]{
  Ziang Chen$^\ast$
}
\author[1]{
  Paata Ivanisvili$^\dagger$
}
\author[2]{
  Jos\'e Madrid$^\dagger$
}
\author[3]{
  Haozhu Wang$^\ast$
}

\affil[1]{University of California, Irvine, \texttt{pivanisv@uci.edu}}
\affil[2]{Virginia Tech, \texttt{josemadrid@vt.edu}}
\affil[3]{xAI, \texttt{\{zchen,hwang\}@x.ai}}

\date{}

\begin{document}

\maketitle

\begin{center}
\small
Authors are listed alphabetically.\\
$^\dagger$Mathematical contributor,
$^\ast$Engineering contributor.
\end{center}

\begin{abstract}
Carbery posed a question of sharpened triangle inequalities for families of operators in the Schatten--von Neumann classes $S_p$, $p\geq 2$. He established a weaker form of the desired estimate for even integer values of $p$. In the commutative setting the corresponding sharp inequality (with optimal exponent $p'=\frac{p}{p-1}$) was recently obtained for all integer $p\geq 2$. In the present work we resolve Carbery's question completely in the non-commutative setting: we prove the sharp inequality
$$
\Big\|\sum_{j} T_{j}\Big\|_{S_{p}}\leq \|(\alpha_{ij}^{p'})\|^{1/p'}_{\ell_{2}\to \ell_{2}} \Big( \sum_{j} \|T_{j}\|^{p}_{S_{p}}\Big)^{1/p}
$$
for \emph{all} $p\geq 2$ and  all countable sequences of operators $(T_{j}) \subset S_{p}$, where $\alpha_{ij}$ are almost orthogonality coefficients. The proof is based on a block-operator reduction and a  complex interpolation of the
polar parts of the blocks. 
\end{abstract}

\section{Introduction}

Let $\mathcal{H}$ be a separable complex Hilbert space and let $S_p = S_p(\mathcal{H})$ denote the Schatten--von Neumann class of compact operators on $\mathcal{H}$ equipped with the Schatten norm $\|\cdot\|_{S_p}$, $p\geq 2$.
Recall that for $T \in S_{p}$ we have $\|T\|_{S_{p}} = \Big(\sum_{j} \lambda_{j}^{p}\Big)^{1/p
}$ where $(\lambda_{j})$ are singular values of $T$, i.e.,  the eigenvalues of $|T|:=\sqrt{T^{*}T}$. For $p=\infty$ we identify the space $S_{\infty}$ with all bounded linear operators on $\mathcal{H}$ with the usual operator norm. We refer the reader to \cite{Ringrose1971} for basic properties of $S_{p}$ classes.

Carbery  \cite{Carbery2009}  initiated the study of almost-orthogonality phenomena in these classes. Given a finite family $T_1,\dots,T_N\in S_p$, he introduced the symmetric matrix of orthogonality coefficients  $(\alpha_{ij})_{i,j=1}^{N}$ with entries $\alpha_{ij} \in [0,1]$  satisfying
\begin{align}\label{carbp}
\|T_i T^{*}_j\|_{S_{p/2}} \leq \alpha_{ij}^{2} \|T_i\|_{S_p}\|T_j\|_{S_p} \quad \text{and} \quad \|T^{*}_i T_j\|_{S_{p/2}} \leq \alpha_{ij}^{2} \|T_i\|_{S_p}\|T_j\|_{S_p}
\end{align}
and asked for sharpened forms of the triangle inequality measuring the almost-orthogonality through the coefficients $\alpha_{ij}$. In particular, in \cite[Section~3.3, (3.1)]{Carbery2009}, he asked whether the inequality
\begin{align}\label{enisori}
\Bigl\|\sum_{j=1}^N T_j\Bigr\|_{S_p}
\leq
\|(\alpha_{ij}^{p'})\|_{2}^{1/p'}
\Bigl(\sum_{j=1}^N\|T_j\|_{S_p}^p\Bigr)^{1/p}
\end{align}
holds true for all $p\geq 2$, where $\| (\alpha_{ij}^{p'})\|_{q}$ denotes the operator norm from $\ell_{q}^{N}$ to $\ell_{q}^{N}$. He proved \cite[Theorem~1.3]{Carbery2009} a weaker estimate with the Schur norm $\|(\alpha_{ij})\|_{1}^{1/p'}$ instead of  $\|(\alpha_{ij}^{p'})\|_{2}^{1/p'}$, but only when $p$ is an even integer. Carbery also asked a related question \cite[Section~3.5]{Carbery2009} whether the inequality 
\begin{align}\label{oripower}
\Bigl\|\sum_{j=1}^N T_j\Bigr\|_{S_p}
\leq
\|(\alpha_{ij}^{2})\|_{1}^{1/p'}
\Bigl(\sum_{j=1}^N\|T_j\|_{S_p}^p\Bigr)^{1/p}
\end{align}
holds for all $N\geq 2$ in the commutative case. 

In the commutative $L^p$ setting, the first sharp results concerned the two-function case. For $N=2$, Carbery \cite[Proposition~3.1]{Carbery2009} established the analogue of \eqref{oripower} for characteristic functions of sets. Carlen--Frank--Ivanisvili--Lieb \cite{CFIL2021} later proved a substantially stronger two-function inequality for arbitrary functions in $L^p$; in particular, their result implies Carbery's proposed estimate (\ref{oripower}) when $N=2$. This two-function result was subsequently refined by Ivanisvili--Mooney \cite{IvanisviliMooney2020}. 

For arbitrary finite families, the problem is more delicate. A recent work \cite{CDPIMW2026} constructs a counterexample showing that no estimate of the form (\ref{oripower}) can hold, in general, with any power on the coefficients $\alpha_{ij}$ strictly larger than $p'$. The same work proves the critical exponent $p'$ for every integer $p\geq 2$ in the commutative $L^{p}$ setting. Its proof expands the integer power and combines the resulting multilinear expression with a Schur-type test. Similar inequalities for many functions have been studied also in \cite{CFL2020}.

In this paper we  affirmatively answer the question posed by Carbery in \cite[Section~3.3, (3.1)]{Carbery2009} in full generality for the Schatten--von Neumann classes. We prove the following sharp statement, valid for \emph{all} real $p\geq 2$ and for \emph{arbitrary} (not necessarily positive or self-adjoint) operators in $S_{p}$. 

\begin{theorem}\label{mainthm}
Let $N\geq 2$,  $p\geq 2$,  and let $T_1,\dots,T_N\in S_p(\mathcal{H})$. Let $(\alpha_{ij})_{1\leq i,j\leq N}$ be a symmetric matrix with entries $\alpha_{ij} \in [0,1]$ satisfying (\ref{carbp}). Then
\[
\Bigl\|\sum_{j=1}^N T_j\Bigr\|_{S_p}
\leq
\bigl\|(\alpha_{ij}^{p'})\bigr\|_{2}^{1/p'}
\Bigl(\sum_{j=1}^N\|T_j\|_{S_p}^p\Bigr)^{1/p}.
\]
\end{theorem}

\begin{remark}
    Since
\[
   \norm{(\alpha_{ij}^{p'})}_{2}
   \le
   \sqrt{\norm{(\alpha_{ij}^{p'})}_{1} \norm{(\alpha_{ij}^{p'})}_{\infty}} = \|(\alpha^{p'}_{ij})\|_{1}=
   \sup_i\sum_j \alpha^{p'}_{ij}\leq \|(\alpha_{ij})\|_{1},
\]
Theorem~\ref{mainthm} improves Carbery’s even-integer Schur-$(\alpha_{ij})$ estimate,  answers the $\ell_{2}$-operator-norm question (\ref{enisori}), and it extends \cite[Theorem~1.1]{CDPIMW2026}  to Schatten classes and all $p\geq 2$. 
\end{remark}

\begin{remark}
    The exponent $p'$ in the constant  $\norm{(\alpha_{ij}^{p'})}_{2}$ is optimal, i.e.,  the counter-example from \cite{CDPIMW2026}
   shows that no larger exponent than $p'$ is admissible on the coefficients $\alpha_{ij}$.
\end{remark}

\begin{remark}
    When $p=\infty$ the result should be understood as follows 
\[
\Bigl\|\sum_{j=1}^N T_j\Bigr\|_{\infty}
\leq
\bigl\|(\alpha_{ij})\bigr\|_{2} \sup_{j} 
\|T_j\|_{\infty},
\]
where symmetric  $\alpha_{ij} \in [0,1]$ satisfy $\| T_{i}T_{j}^{*}\|_{\infty}\leq \alpha_{ij}^{2}\|T_{i}\|_{\infty}\|T_{j}\|_{\infty}$ and $\| T_{i}^{*}T_{j}\|_{\infty}\leq \alpha_{ij}^{2}\|T_{i}\|_{\infty}\|T_{j}\|_{\infty}$. Here $\| T\|_{\infty}$ denotes the operator norm of a bounded operator on $\mathcal{H}$. The proof of this endpoint case is given in Section~\ref{Grok11}.
\end{remark}

\begin{remark}
    Passage from a finite to a countable family $(T_{j})$ is given in Section~\ref{count1}.
\end{remark}

The proof proceeds in two steps. First we establish the result for finite positive operators by means of complex interpolation between the Schatten classes $S_2$ and $S_\infty$. The analytic family is constructed so that the scalar factors $\alpha_{ij}$ compensate for the growth of the blocks $X_{ij}$ (see Section~\ref{compli}) on the left boundary while keeping the right boundary under control by the operator norm of the matrix $(\alpha_{ij}^{p'})$. The case $p=2$ is handled directly by the Hilbert--Schmidt inner-product argument. The passage from positive to arbitrary operators is done by the standard $2\times 2$ positive dilation: each $T_j$ is replaced by the positive operator
\[
P_j=\begin{pmatrix}|T_j^*|&T_j\\T_j^*&|T_j|\end{pmatrix}
\]
on $\mathcal{H}\oplus\mathcal{H}$. A simple lemma on off-diagonal blocks of positive $2\times 2$ matrices then reduces the original estimate to the positive case already proved, while a direct computation verifies that the same coefficients $\alpha_{ij}$ remain admissible for the dilated family.

The interpolation technique employed here is in the spirit of the abstract complex interpolation theory for operator-valued kernels developed by Fournier and Russo \cite{FournierRusso1977}, although we work directly with the Schatten classes on a fixed Hilbert space rather than with integral operators. The necessary interpolation result for Schatten classes---namely $[S_2,S_\infty]_\theta=S_p$ with $\theta=1-2/p$---is classical and may be found, for example, in Simon \cite{SimonTraceIdeals}. 

The paper is organized as follows. Section 2 contains the complete proof of Theorem~\ref{mainthm}. We first treat the finite positive case (Subsection 2.1 for $p=2$ and Subsection 2.2 for $p>2$), and then reduce the general case to the positive case via the $2\times 2$ dilation (Subsection 2.3).

\section{Proof of Theorem~\ref{mainthm}}\label{compli}

We first prove the estimate for positive finite-rank operators. The extension
from finite-rank positive operators to arbitrary positive operators in $S_p$
is obtained by approximation at the end of Subsection~\ref{subsec:positive-pgtwo}.
The general case is then reduced to the positive
case in Subsection~\ref{subsec:general}.

For a positive family $T_1,\dots,T_N$, define
\[
R:\mathcal H^N\to\mathcal H,\qquad
R(h_1,\dots,h_N)=\sum_{j=1}^N T_j^{1/2}h_j,
\]
and set $X=R^*R$. Then the $(i,j)$-block of $X$ is
\[
X_{ij}=T_i^{1/2}T_j^{1/2},
\]
while
\[
RR^*=\sum_{j=1}^N T_j.
\]
Therefore $X$ and $\sum_j T_j$ have the same non-zero singular values, and hence
\[
\|X\|_{S_q(\mathcal H^N)}
=
\Bigl\|\sum_{j=1}^N T_j\Bigr\|_{S_q(\mathcal H)}
\qquad (1\le q\le \infty).
\]

\subsection{Positive case: $p=2$} 

Let $a_j=\|T_j\|_{S_2}$, and let $X=R^*R$ be the block operator defined above. 
We have  $\|X_{ij}\|_{S_2}^2\leq\|T_i T_j\|_{S_1}$. Indeed,
\[
\|X_{ij}\|_{S_2}^2
=\tr(T_j^{1/2}T_iT_j^{1/2})
=\tr(T_iT_j)
\le \|T_iT_j\|_{S_1}.
\]
Therefore
\[
\|X_{ij}\|_{S_2}\leq\alpha_{ij}\,a_i^{1/2}a_j^{1/2}.
\]
The Hilbert--Schmidt norm of the block operator therefore satisfies
\[
\|X\|_{S_2}^2=\sum_{i,j}\|X_{ij}\|_{S_2}^2\leq\sum_{i,j}\alpha_{ij}^2 a_i a_j.
\]
Let $\Delta=(\alpha_{ij}^2)$. The right-hand side is the quadratic form $\langle\Delta a,a\rangle$ evaluated at the vector $a=(a_1,\dots,a_N)$. Hence it is at most $K\|a\|_{\ell_2}^2=K\sum a_j^2$, where $K=\|(\alpha_{ij}^2)\|_{\ell_2\to\ell_2}$. It follows that
\[
\|X\|_{S_2}\leq K^{1/2}\Bigl(\sum_j a_j^2\Bigr)^{1/2},
\]
which is the asserted bound for $p=2$. The same argument is valid for
arbitrary positive operators in $S_2$. 

\subsection{Positive case: $p>2$ (complex interpolation)}
\label{subsec:positive-pgtwo}

Let $a_j=\|T_j\|_{S_p}$, and let $X=R^*R$ be the block operator defined above.
Let $q=p/2$. Since
\[
X_{ij}^*X_{ij}=T_j^{1/2}T_iT_j^{1/2},
\]
the non-zero eigenvalues of $X_{ij}^*X_{ij}$ agree, with algebraic
multiplicity, with those of $T_j^{1/2}T_j^{1/2}T_i=T_jT_i$, and hence with
those of $T_iT_j$. Since $X_{ij}^*X_{ij}$ is positive, its $S_q$-norm is the
$\ell_q$-norm of this eigenvalue list. Weyl's majorant theorem, applied to the
possibly non-normal operator $T_iT_j$, gives
\[
\sum_k \lambda_k(X_{ij}^*X_{ij})^q
\le
\sum_k s_k(T_iT_j)^q,
\]
where $\lambda_k(X_{ij}^*X_{ij})$ are the non-zero eigenvalues of
$X_{ij}^*X_{ij}$ and $s_k(T_iT_j)$ are the singular values of $T_iT_j$.
Therefore
\[
\|X_{ij}\|_{S_p}^2
=
\|X_{ij}^*X_{ij}\|_{S_q}
\le
\|T_iT_j\|_{S_q}.
\]
Thus, by the hypothesis,
\[
\|X_{ij}\|_{S_p}^2\leq \alpha_{ij}^2 a_i a_j
\]
and therefore $\|X_{ij}\|_{S_p}\leq\alpha_{ij}\,a_i^{1/2}a_j^{1/2}$. Write $X_{ij}=U_{ij}|X_{ij}|$ for the polar decomposition. Define the analytic family of block operators on the strip $0\leq\operatorname{Re}z\leq 1$ by declaring that if $\alpha_{ij}>0$,
\[
F(z)_{ij}=\alpha_{ij}^{\phi(z)}\,U_{ij}\,|X_{ij}|^{\psi(z)},
\]
where
\[
\psi(z)=\frac{p}{2}(1-z),\qquad\phi(z)=\Bigl(\frac{p'}{2}-\frac{p}{2}\Bigr)\frac{\theta-z}{\theta}
\]
with interpolation point $\theta=1-2/p\in(0,1)$. If \(\alpha_{ij}=0\), then the preceding estimate gives
\(\|X_{ij}\|_{S_p}=0\), hence \(X_{ij}=0\); in this case set
\(F(z)_{ij}=0\) for all \(z\). Since we are first working with finite-rank positive operators, this family is
defined by finite-dimensional spectral calculus on $\operatorname{supp}|X_{ij}|$,
and it is set equal to $0$ on $\ker |X_{ij}|$. Thus $z\mapsto F(z)$ is analytic
in the open strip and continuous on the closed strip as a function with values in
$S_2+S_\infty$. By construction $\psi(\theta)=1$ and $\phi(\theta)=0$, so $F(\theta)=X$.

By the standard complex interpolation identity for Schatten ideals,
\[
[S_2(\mathcal H^N),S_\infty(\mathcal H^N)]_\theta
=
S_p(\mathcal H^N),
\qquad \theta=1-\frac2p,
\]
see, for example, \cite{SimonTraceIdeals}, we obtain
\[
\|X\|_{S_p}\leq\Bigl(\sup_t\|F(it)\|_{S_2}\Bigr)^{2/p}\Bigl(\sup_t\|F(1+it)\|_\infty\Bigr)^{(p-2)/p}.
\]

On the left boundary, for blocks with $\alpha_{ij}>0$ the real part of the exponent on $|X_{ij}|$ is $p/2$ while the imaginary part contributes only a unitary factor. The real part of the exponent on $\alpha_{ij}$ is $p'/2-p/2$. Hence
\[
\|F(it)_{ij}\|_{S_2}=\alpha_{ij}^{p'/2-p/2}\|X_{ij}\|_{S_p}^{p/2}\leq\alpha_{ij}^{p'/2}a_i^{p/4}a_j^{p/4}.
\]
Summing the squared block norms gives
\[
\|F(it)\|_{S_2}^2\leq\sum_{i,j}\alpha_{ij}^{p'}a_i^{p/2}a_j^{p/2}\leq K\sum_k a_k^p,
\]
where $K=\|(\alpha_{ij}^{p'})\|_{\ell_2\to\ell_2}$. Thus $\|F(it)\|_{S_2}\leq(K\sum a_k^p)^{1/2}$, and raising to the power $2/p$ produces the factor $K^{1/p}(\sum a_j^p)^{1/p}$.

On the right boundary $z=1+it$ we have
\[
\operatorname{Re}\psi(1+it)=0,\qquad \operatorname{Re}\phi(1+it)=p'.
\]
Hence each nonzero block can be written as
\[
F(1+it)_{ij}=\alpha_{ij}^{p'} C_{ij}(t),
\qquad \|C_{ij}(t)\|_{\infty}\le 1,
\]
and the blocks with $\alpha_{ij}=0$ are zero. Therefore, for
$h=(h_1,\dots,h_N)\in \mathcal H^N$,
\[
\|(F(1+it)h)_k\|
\le
\sum_{j=1}^N \alpha_{kj}^{p'}\|h_j\|.
\]
Taking the $\ell_2$-norm in $k$ gives
\[
\|F(1+it)h\|_{\mathcal H^N}
\le
\|(\alpha_{ij}^{p'})\|_{\ell_2^N\to\ell_2^N}
\bigl\|(\|h_j\|)_{j=1}^N\bigr\|_{\ell_2^N}
=
K\|h\|_{\mathcal H^N}.
\]
Thus
\[
\|F(1+it)\|_\infty\le K.
\]
Raising to the power $(p-2)/p$ contributes the factor $K^{(p-2)/p}$.

Combining the boundary estimates yields
\[
\|X\|_{S_p}\leq K^{1/p+(p-2)/p}\Bigl(\sum_j a_j^p\Bigr)^{1/p}=K^{1/p'}\Bigl(\sum_j a_j^p\Bigr)^{1/p}.
\]
This proves the desired estimate for positive finite-rank operators.

It remains to remove the finite-rank assumption. Let $T_j\ge0$ be arbitrary
operators in $S_p$. After discarding the zero terms, choose positive finite-rank
operators $T_j^{(n)}$ such that $T_j^{(n)}\to T_j$ in $S_p$. By Hölder's
inequality,
\[
\|T_i^{(n)}T_j^{(n)}-T_iT_j\|_{S_{p/2}}
\le
\|T_i^{(n)}-T_i\|_{S_p}\|T_j^{(n)}\|_{S_p}
+
\|T_i\|_{S_p}\|T_j^{(n)}-T_j\|_{S_p},
\]
and hence $T_i^{(n)}T_j^{(n)}\to T_iT_j$ in $S_{p/2}$. Fix $\varepsilon>0$ and set
\[
\beta_{ij}^{(\varepsilon)}=\min\{1,\alpha_{ij}+\varepsilon\}.
\]
For all sufficiently large $n$, the finite-rank family $(T_j^{(n)})$ satisfies
the positive-case hypotheses with coefficients $\beta_{ij}^{(\varepsilon)}$.
Applying the finite-rank estimate and then letting $n\to\infty$ gives
\[
\Bigl\|\sum_{j=1}^N T_j\Bigr\|_{S_p}
\le
\|((\beta_{ij}^{(\varepsilon)})^{p'})\|_{\ell_2\to\ell_2}^{1/p'}
\Bigl(\sum_{j=1}^N\|T_j\|_{S_p}^p\Bigr)^{1/p}.
\]
Finally let $\varepsilon\downarrow0$. Since $N<\infty$, the matrix norm is
continuous, and the desired positive-operator inequality follows.

\subsection{General case (arbitrary operators)}
\label{subsec:general}

Let $T_j=U_j|T_j|$ be the polar decomposition and define the positive operators on the doubled space $\mathcal{H}\oplus\mathcal{H}$ by
\begin{align}\label{dadeb}
P_j=\begin{pmatrix}|T_j^*|&T_j\\T_j^*&|T_j|\end{pmatrix}.
\end{align}
Let $s_j$ be the support projection of $|T_j|$. Since $U_j^*U_j=s_j$, the map
\[
W_j:s_j\mathcal H\to \mathcal H\oplus\mathcal H,\qquad
W_jh=2^{-1/2}(U_jh,h),
\]
is an isometry. Moreover,
\[
P_j=2W_j|T_j|W_j^*.
\]
Thus $P_j\ge0$, and the nonzero eigenvalues of $P_j$ are exactly twice the
singular values of $T_j$. Consequently
\[
\|P_j\|_{S_p}=2\|T_j\|_{S_p}=2a_j.
\]
Moreover,
\[
\sum_j P_j=\begin{pmatrix}\sum_j|T_j^*|&\sum_j T_j\\\sum_j T_j^*&\sum_j|T_j|\end{pmatrix}.
\]

\begin{lemma}\label{lem:offdiag}
Let
\[
M=\begin{pmatrix}A&B\\B^*&C\end{pmatrix}\ge0
\]
belong to $S_p(\mathcal H\oplus\mathcal H)$, where $p\ge1$. Then
\[
2\|B\|_{S_p}\le \|M\|_{S_p}.
\]
\end{lemma}

\begin{proof}
Let $s_k(B)$ be the singular values of $B$, with singular-vector pairs
$(e_k,f_k)$ chosen so that $Bf_k=s_k(B)e_k$ and $B^*e_k=s_k(B)f_k$. Set
\[
w_k=\frac{1}{\sqrt2}(e_k,f_k),\qquad
v_k=\frac{1}{\sqrt2}(e_k,-f_k).
\]
Since $M\ge0$, $\langle Mv_k,v_k\rangle\ge0$, and
\[
\langle Mw_k,w_k\rangle-\langle Mv_k,v_k\rangle=2s_k(B).
\]
Hence $\langle Mw_k,w_k\rangle\ge 2s_k(B)$. By Ky Fan's maximum principle, see e.g. \cite[Chapter III]{BhatiaMatrixAnalysis}, 
for every $m$,
\[
2\sum_{k=1}^m s_k(B)
\le
\sum_{k=1}^m\langle Mw_k,w_k\rangle
\le
\sum_{k=1}^m \lambda_k(M).
\]
Thus $2s(B)\prec_w \lambda(M)$. Taking $\ell_p$-norms gives the result.
\end{proof}
Applying Lemma~\ref{lem:offdiag} to $M=\sum_j P_j$ and to its off-diagonal
block $\sum_j T_j$ gives
\[
2\Bigl\|\sum_j T_j\Bigr\|_{S_p}\leq\Bigl\|\sum_j P_j\Bigr\|_{S_p}.
\]

It remains to verify that the same coefficients $(\alpha_{ij})$ are admissible
for the positive family $(P_j)$. Since $P_i=P_i^*$ and $P_j=P_j^*$, the two
conditions in \eqref{carbp} coincide for this family; hence it is enough to
prove
\[
\|P_iP_j\|_{S_{p/2}}
\le
\alpha_{ij}^2\|P_i\|_{S_p}\|P_j\|_{S_p}.
\]
Set
\[
V_j:\mathcal H\to\mathcal H\oplus\mathcal H,\qquad V_jh=(U_jh,h).
\]
Then $\|V_j\|_\infty\le \sqrt2$, $V_i^*V_j=U_i^*U_j+I$, and
\[
P_j=V_j|T_j|V_j^*.
\]
Indeed, $|T_j|=s_j|T_j|$ and $U_j^*U_j=s_j$. Hence
\[
P_iP_j
=
V_i|T_i|V_i^*V_j|T_j|V_j^*
=
V_i|T_i|(U_i^*U_j+I)|T_j|V_j^*.
\]
By the ideal property of Schatten norms (i.e., $\|AXB\|_{S_r}
\le
\|A\|_{\infty}\,\|X\|_{S_r}\,\|B\|_{\infty}$, $1\le r\le \infty$, 
whenever \(X\in S_r\) and \(A,B\) are bounded operators),  and $\|V_i\|_\infty,\|V_j\|_\infty\le
\sqrt2$,
\[
\|P_iP_j\|_{S_{p/2}}
\le
2\bigl\||T_i|(U_i^*U_j+I)|T_j|\bigr\|_{S_{p/2}}.
\]
The triangle inequality splits the right-hand side. The first summand is
\[
|T_i|U_i^*U_j|T_j|=T_i^*T_j,
\]
and therefore
\[
\bigl\||T_i|U_i^*U_j|T_j|\bigr\|_{S_{p/2}}
=
\|T_i^*T_j\|_{S_{p/2}}.
\]
For the second summand, put $A=|T_i||T_j|$. Then
$A=s_iAs_j$, where $s_i$ and $s_j$ are the support projections of $|T_i|$
and $|T_j|$. Let $t_j=U_jU_j^*$ be the support projection of $|T_j^*|$. Since
$U_i:s_i\mathcal H\to U_iU_i^*\mathcal H$ and
$U_j:s_j\mathcal H\to t_j\mathcal H$ are unitary maps between their support
spaces, the operators $A$ and $U_iAU_j^*$ have the same non-zero singular
values. Equivalently,
\[
(U_iAU_j^*)^*(U_iAU_j^*)
=
U_jA^*U_i^*U_iAU_j^*
=
U_jA^*AU_j^*,
\]
and $U_jA^*AU_j^*$ is unitarily equivalent to $A^*A$ on the support $s_j\mathcal
H$. Since $U_iAU_j^*=T_iT_j^*$, it follows that
\[
\bigl\||T_i||T_j|\bigr\|_{S_{p/2}}
=
\|T_iT_j^*\|_{S_{p/2}}.
\]
By the given hypothesis both quantities are at most $\alpha_{ij}^2 a_i a_j$, so
\[
\|P_i P_j\|_{S_{p/2}}\leq 4\alpha_{ij}^2 a_i a_j.
\]
On the other hand $\|P_i\|_{S_p}\|P_j\|_{S_p}=4a_i a_j$, and therefore
\[
\|P_i P_j\|_{S_{p/2}}\leq\alpha_{ij}^2\|P_i\|_{S_p}\|P_j\|_{S_p}.
\]
The positive-operator inequality already established (applied to the family $(P_j)$) yields
\[
\Bigl\|\sum_j P_j\Bigr\|_{S_p}\leq K^{1/p'}\Bigl(\sum_j\|P_j\|_{S_p}^p\Bigr)^{1/p}=2K^{1/p'}\Bigl(\sum_j a_j^p\Bigr)^{1/p},
\]
where $K=\|(\alpha_{ij}^{p'})\|_{\ell_2\to\ell_2}$. Combining with the off-diagonal estimate and cancelling the factor of 2 produces
\[
\Bigl\|\sum_{j=1}^N T_j\Bigr\|_{S_p}\leq K^{1/p'}\Bigl(\sum_{j=1}^N a_j^p\Bigr)^{1/p}.
\]
This is the asserted inequality for arbitrary operators. The zero cases are covered by the conventions and approximation argument above.

This completes the proof of Theorem~\ref{mainthm}.
\subsection{Concluding remarks}
\subsubsection{Countable family}\label{count1}
Theorem~\ref{mainthm} extends in the same form to countable families. More
precisely, let \(2\le p<\infty\), let \((T_j)_{j\ge1}\subset S_p(\mathcal H)\)
satisfy
\[
\sum_{j=1}^{\infty}\|T_j\|_{S_p}^p<\infty,
\]
and let \((\alpha_{ij})_{i,j\ge1}\) be a symmetric matrix with entries in
\([0,1]\) satisfying \eqref{carbp} for every pair \(i,j\). Assume moreover that
\[
\bigl\|(\alpha_{ij}^{p'})_{i,j\ge1}\bigr\|_{\ell_2(\mathbb N)\to
\ell_2(\mathbb N)}<\infty.
\]
Then the series \(\sum_{j=1}^{\infty}T_j\) converges in \(S_p\), and
\[
\Bigl\|\sum_{j=1}^{\infty}T_j\Bigr\|_{S_p}
\le
 \bigl\|(\alpha_{ij}^{p'})_{i,j\ge1}\bigr\|_{\ell_2(\mathbb N)\to
\ell_2(\mathbb N)}^{1/p'}
\Bigl(\sum_{j=1}^{\infty}\|T_j\|_{S_p}^p\Bigr)^{1/p}.
\]

Indeed, applying Theorem~\ref{mainthm} to the finite family
\(T_m,\dots,T_n\), and observing that the corresponding finite $(n-m+1)\times (n-m+1)$ coefficient
matrix is a compression of \((\alpha_{ij}^{p'})_{i,j\ge1}\), we get
\[
\Bigl\|\sum_{j=m}^{n}T_j\Bigr\|_{S_p}
\le
\bigl\|(\alpha_{ij}^{p'})_{i,j\ge1}\bigr\|_{\ell_2(\mathbb N)\to
\ell_2(\mathbb N)}^{1/p'}
\Bigl(\sum_{j=m}^{n}\|T_j\|_{S_p}^p\Bigr)^{1/p}.
\]
The right-hand side tends to \(0\) as \(m,n\to\infty\), hence the partial sums
are Cauchy in \(S_p\) (recall that $S_{p}$ is complete). Letting \(n\to\infty\) in the finite estimate for
\(\sum_{j=1}^{n}T_j\) gives the asserted countable-family inequality.

\subsubsection{Endpoint case \texorpdfstring{$p=\infty$}{p=infinity}}\label{Grok11}
We establish the asserted inequality in the endpoint case \(p=\infty\). Let \(T_1,\dots,T_N\in B(\mathcal{H})\) satisfy the almost-orthogonality conditions \eqref{carbp} with coefficients \((\alpha_{ij})\) and with \(a_j=\|T_j\|_\infty\). (As usual we interpret \(S_\infty\) as the space of all bounded operators equipped with the operator norm.) We first treat the positive case and then reduce the general case to it by means of the \(2\times 2\) dilation already employed in Subsection~\ref{subsec:general}.

Assume first that each \(T_j\ge 0\). Define the block operator \(X=R^*R\) on \(\mathcal{H}^N\) precisely as at the beginning of Section~\ref{compli}; since \(X=R^*R\) and \(RR^*=\sum_jT_j\), we have
\[
\|X\|_\infty=\|R^*R\|_\infty=\|R\|_\infty^2
=\|RR^*\|_\infty
=\Bigl\|\sum_{j=1}^N T_j\Bigr\|_\infty .
\]
Moreover,
\[
\|X_{ij}\|_\infty^2
=
\|T_j^{1/2}T_iT_j^{1/2}\|_\infty
=
r(T_j^{1/2}T_iT_j^{1/2})
=
r(T_iT_j)
\le
\|T_iT_j\|_\infty
\le
\alpha_{ij}^2a_ia_j, 
\]
where $r(A)$ denotes the spectral radius of a bounded operator $A$.
Consequently, for an arbitrary vector \(h=(h_1,\dots,h_N)\in\mathcal{H}^N\),
\[
\|(Xh)_i\|\le\Bigl(\max_k a_k\Bigr)\sum_{j=1}^N\alpha_{ij}\|h_j\|.
\]
Taking the \(\ell_2\)-norm over the components \(i\) and invoking the definition of the operator norm of the matrix \((\alpha_{ij})\) on \(\ell_2^N\) gives
\[
\|Xh\|_{\mathcal{H}^N}\le\Bigl(\max_k a_k\Bigr)\bigl\|(\alpha_{ij})\bigr\|_{\ell_2\to\ell_2}\|h\|_{\mathcal{H}^N}.
\]
Hence
\[
\Bigl\|\sum_{j=1}^N T_j\Bigr\|_\infty\le\bigl\|(\alpha_{ij})\bigr\|_{\ell_2\to\ell_2}\cdot\max_j a_j,
\]
which is the desired endpoint estimate for positive operators.

For arbitrary (not necessarily positive or self-adjoint) operators we apply the same \(2\times 2\) positive dilation $P_{j}$
on \(\mathcal{H}\oplus\mathcal{H}\) that was used in Subsection~\ref{subsec:general} (see~(\ref{dadeb})). The identical algebraic computation performed there shows that the family \((P_j)\) satisfies the almost-orthogonality conditions with the \emph{same} coefficients \((\alpha_{ij})\) (the factor \(4\) obtained from the ideal property, the bounds
\(\|V_j\|_\infty\le \sqrt2\), and the two resulting summands cancel exactly
against the factor \(4\) coming from \(\|P_i\|_\infty\|P_j\|_\infty=4a_ia_j\)). The positive case already established therefore yields
\[
\Bigl\|\sum_j P_j\Bigr\|_\infty\le\bigl\|(\alpha_{ij})\bigr\|_{\ell_2\to\ell_2}\cdot\max_j\|P_j\|_\infty=2\bigl\|(\alpha_{ij})\bigr\|_{\ell_2\to\ell_2}\cdot\max_j a_j.
\]

 Finally, the off-diagonal block estimate also holds in operator norm. Indeed,
if
\[
M=\begin{pmatrix}A&B\\B^*&C\end{pmatrix}\ge0,
\]
then \(2\|B\|_\infty\le \|M\|_\infty\). To see this, let \(x,y\in\mathcal H\)
be unit vectors and choose a phase \(\omega\in\mathbb T\) such that
\(\omega\langle By,x\rangle=|\langle By,x\rangle|\). Put
\[
w=\frac1{\sqrt2}(x,\omega y),\qquad
v=\frac1{\sqrt2}(x,-\omega y).
\]
Then \(w\) and \(v\) are unit vectors, and
\[
\langle Mw,w\rangle-\langle Mv,v\rangle
=
2|\langle By,x\rangle|.
\]
Since \(M\ge0\), we have \(\langle Mv,v\rangle\ge0\), hence
\[
2|\langle By,x\rangle|
\le
\langle Mw,w\rangle
\le
\|M\|_\infty .
\]
Taking the supremum over unit \(x,y\) gives
\[
2\|B\|_\infty\le \|M\|_\infty .
\]
 
 Applying this comparison to the off-diagonal block of \(\sum_j P_j\) therefore produces
\[
\Bigl\|\sum_j T_j\Bigr\|_\infty\le\tfrac12\Bigl\|\sum_j P_j\Bigr\|_\infty\le\bigl\|(\alpha_{ij})\bigr\|_{\ell_2\to\ell_2}\cdot\max_j a_j,
\]
as claimed. (Since \(N<\infty\) no approximation argument is required.)

In other words, the inequality of Theorem~\ref{mainthm} continues to hold for \(p=\infty\) upon interpreting the right-hand side in the natural limiting sense
\[
\bigl\|(\alpha_{ij}^{p'})\bigr\|_2^{1/p'}\Bigl(\sum_j\|T_j\|_{S_p}^p\Bigr)^{1/p}\ \longrightarrow\ \bigl\|(\alpha_{ij})\bigr\|_{\ell_2\to\ell_2}\cdot\max_j\|T_j\|_\infty.
\]

\section*{Use of artificial intelligence tools}

The authors have benefited from the use of artificial intelligence tools, most notably Grok. The reduction to the block-operator family \(X\) was proposed by Grok. The complex-interpolation argument was developed through a sustained collaboration between the human authors and the AI system; the corresponding discussion is recorded at
\url{https://grok.com/share/c2hhcmQtNA_23bee943-d863-4c27-803e-e80e7e50082d}. The proof presented in this paper is a streamlined and adapted version of that argument and has been extended to arbitrary operators via the standard \(2\times 2\) positive dilation. We also note that a proof of the commutative \(L^p\) analog of (\ref{oripower}) with optimal $(\alpha_{ij}^{p'})$ instead of $(\alpha_{ij}^{2})$  was obtained by a different route (not recorded in this paper), in which one essential step originated with a human author and the remaining steps were due to Grok (see
\url{https://grok.com/share/c2hhcmQtNA_ba3f087b-b0e7-40b2-815e-d9ab7de6f8f5}). Section~\ref{Grok11}, i.e., the endpoint $p=\infty$ case,  is the adapted and polished version of a response due to Grok:  \url{https://grok.com/share/c2hhcmQtNA_640ef40f-ffba-4775-a00a-6aa486cbf074}.

\section*{Acknowledgments}

P.I. acknowledges partial support from the NSF CAREER grant DMS-2152401, NSF grant DMS- 2554183, a Simons Fellowship, and a Humboldt Research Fellowship for Experienced Researchers. J.M. was partially supported by the Simons Foundation Grant $\# 453576$. 

\bibliographystyle{plain}

\begin{thebibliography}{10}

\bibitem{BhatiaMatrixAnalysis}
R.~Bhatia.
\newblock \emph{Matrix Analysis}.
\newblock Graduate Texts in Mathematics 169, Springer, 1997.

\bibitem{Carbery2009}
A. Carbery.
\newblock Almost-orthogonality in the Schatten--von Neumann classes.
\newblock \emph{J. Operator Theory}, 62(1):151--158, 2009.

\bibitem{CFIL2021}
E.~A. Carlen, R.~L. Frank, P. Ivanisvili, and E.~H. Lieb.
\newblock Inequalities for $L^p$-norms that sharpen the triangle inequality and complement Hanner's inequality.
\newblock \emph{J. Geom. Anal.}, 31:4051--4073, 2021.

\bibitem{CFL2020}
E.~A. Carlen, R.~L. Frank, and E.~H. Lieb.
\newblock Inequalities that sharpen the triangle inequality for sums of $N$ functions in $L^p$.
\newblock \emph{Ark. Mat.}, 58(1):57--69, 2020.


\bibitem{CDPIMW2026}
Z. Chen, J. de Dios Pont, P. Ivanisvili, J. Madrid, and H. Wang.
\newblock Almost-orthogonality in $L^p$ spaces: a case study with Grok.
\newblock \emph{arXiv preprint arXiv:2605.05192}, 2026.


\bibitem{FournierRusso1977}
J.~J.~F. Fournier and B. Russo.
\newblock Abstract interpolation and operator-valued kernels.
\newblock \emph{J. London Math. Soc. (2)}, 16(2):283--289, 1977.

\bibitem{IvanisviliMooney2020}
P. Ivanisvili and C. Mooney.
\newblock Sharpening the triangle inequality: envelopes between $L^2$ and $L^p$ spaces.
\newblock \emph{Analysis \& PDE}, 13(5):1591--1603, 2020.

\bibitem{Ringrose1971}
J.~R. Ringrose.
\newblock \emph{Compact Nonselfadjoint Operators}.
\newblock Van Nostrand Reinhold, London--New York, 1971.

\bibitem{SimonTraceIdeals}
B.~Simon.
\newblock \emph{Trace Ideals and Their Applications}.
\newblock 2nd ed., Mathematical Surveys and Monographs 120,
American Mathematical Society, 2005.


\end{thebibliography}

\end{document}